\documentclass[12pt]{amsart}

\usepackage{psfrag}
\usepackage{color}
\usepackage{tikz}
\usetikzlibrary{matrix,arrows}
\usepackage{graphicx,graphics}
\usepackage{fullpage,amssymb,amsfonts,amsmath,amstext,amsthm,amscd,verbatim}
\usepackage[T1]{fontenc}

\begin{document}

\newtheorem{theorem}{Theorem}[section]
\newtheorem{result}[theorem]{Result}
\newtheorem{fact}[theorem]{Fact}
\newtheorem{conjecture}[theorem]{Conjecture}
\newtheorem{definition}[theorem]{Definition}
\newtheorem{lemma}[theorem]{Lemma}
\newtheorem{proposition}[theorem]{Proposition}
\newtheorem{remark}[theorem]{Remark}
\newtheorem{corollary}[theorem]{Corollary}
\newtheorem{facts}[theorem]{Facts}
\newtheorem{props}[theorem]{Properties}
\newtheorem*{thmA}{Theorem A}
\newtheorem{ex}[theorem]{Example}

\newcommand{\notes} {\noindent \textbf{Notes.  }}
\newcommand{\note} {\noindent \textbf{Note.  }}
\newcommand{\defn} {\noindent \textbf{Definition.  }}
\newcommand{\defns} {\noindent \textbf{Definitions.  }}
\newcommand{\x}{{\bf x}}
\newcommand{\z}{{\bf z}}
\newcommand{\B}{{\bf b}}
\newcommand{\V}{{\bf v}}
\newcommand{\T}{\mathbb{T}}
\newcommand{\Z}{\mathbb{Z}}
\newcommand{\Hp}{\mathbb{H}}
\newcommand{\D}{\mathbb{D}}
\newcommand{\R}{\mathbb{R}}
\newcommand{\N}{\mathbb{N}}
\renewcommand{\B}{\mathbb{B}}
\newcommand{\C}{\mathbb{C}}
\newcommand{\ft}{\widetilde{f}}
\newcommand{\dt}{{\mathrm{det }\;}}
 \newcommand{\adj}{{\mathrm{adj}\;}}
 \newcommand{\0}{{\bf O}}
 \newcommand{\av}{\arrowvert}
 \newcommand{\zbar}{\overline{z}}
 \newcommand{\xbar}{\overline{X}}
 \newcommand{\htt}{\widetilde{h}}
\newcommand{\ty}{\mathcal{T}}
\renewcommand\Re{\operatorname{Re}}
\renewcommand\Im{\operatorname{Im}}
\newcommand{\tr}{\operatorname{Tr}}

\newcommand{\ds}{\displaystyle}
\numberwithin{equation}{section}

\renewcommand{\theenumi}{(\roman{enumi})}
\renewcommand{\labelenumi}{\theenumi}

\title{On B\"{o}ttcher coordinates and quasiregular maps}

\author[Fletcher]{Alastair Fletcher}
\address{University of Warwick \\ Institute of Mathematics \\  Coventry, CV4 7AL, UK}
\email{Alastair.Fletcher@warwick.ac.uk}

\author[Fryer]{Rob Fryer}
\address{University of Warwick \\ Institute of Mathematics \\ Coventry,
CV4 7AL, UK}
\email{fryer.rob@gmail.com}

\begin{abstract}
It is well-known that a polynomial $f(z)=a_dz^d(1+o(1))$ can be conjugated by a holomorphic map $\phi$ to $w \mapsto w^d$ in a neighbourhood of infinity. This map $\phi$ is called a B\"{o}ttcher coordinate for $f$ near infinity.
In this paper we construct a B\"{o}ttcher type coordinate for compositions of
affine mappings and polynomials, a class of mappings first studied in \cite{FG}. As an application, we prove that if $h$ is affine and $c \in \C$, then $h(z)^2+c$ is not uniformly quasiregular.

\end{abstract}

\thanks{The authors wish to thank Vladimir Markovic for stimulating conversations on the subject matter. The first author was supported by EPSRC grant EP/G050120/1.}

\maketitle

\section{Introduction}

In recent years, there has been a great deal of interest in the dynamics of holomorphic functions in the plane, see for example the books of Milnor \cite{Milnor} or Carleson and Gamelin \cite{CG}. However, interest in the iteration of such functions goes back further. The first burst of activity occurred at the beginning of the 20th century with the work of Fatou, Julia and others. Of particular interest here is the following theorem of B\"{o}ttcher from 1904.

\begin{thmA}[\cite{Bottcher}]
Let $f$ be holomorphic in a neighbourhood $U$ of infinity, and let infinity be a superattracting fixed point of $f$, that is, there exists $d \geq 2$ such that
\[ f(z) = a_dz^d ( 1+o(1)),\]
for $z \in U$, where $a_d \in \C \setminus \{ 0 \}$. Then there exists a holomorphic change of coordinate $w = \phi(z)$, with $\phi(\infty) = \infty$, which conjugates $f$ to $w \mapsto w^d$ in some neighbourhood of infinity. Further, $\phi$ is unique up to multiplication by an $(n-1)$'st root of unity.
\end{thmA}

The map $\phi$ is called a {\it B\"{o}ttcher coordinate} for $f$ near infinity. In this article, we will find an analogous B\"{o}ttcher coordinate for mappings of the form $f=g \circ h$, where $g$ is a polynomial of degree $d \geq 2$, and $h$ is an affine mapping of the plane to itself.

Such mappings are examples of quasiregular mappings. Informally, quasiregular mappings send infinitesimal circles to infinitesimal ellipses. The greater the eccentricity of the ellipses, the larger the distortion of the mapping. Quasiregular mappings can be defined in any dimension, see Rickman's monograph \cite{Rickman} for more details. They can be viewed as generalizations of holomorphic mappings in the plane, and share some similar properties which means their dynamics can be studied.

The iterates of a quasiregular mapping are again quasiregular, and were first studied in the special case where there is a uniform bound on the distortion of the iterates, see \cite{IM}. Such mappings are called uniformly quasiregular mappings. These are the closest relatives to holomorphic mappings, and the behaviour of uniformly quasiregular mappings near their fixed points was studied in \cite{HM,HMM}.
It turns out the dynamics of quasiregular mappings can still be studied in the absence of uniform quasiregularity. Recent papers in this direction include \cite{Bergweiler,BE,BFLM,FG,FN}.

We also mention that B\"{o}ttcher coordinates have been studied in the setting of several complex variables \cite{BEK}.

\section{Statement of results}

The purpose of this article is to continue the study of the dynamics of the quasiregular mappings $h(z)^2+c$ initiated in \cite{FG},
where $h$ is an affine mapping and $c \in \C$. In particular, we prove the following analogue of B\"{o}ttcher coordinates for these mappings.

\begin{theorem}
\label{thm1}
Let $h:\C \to \C$ be an affine mapping and $c \in \C$. Then there exists a neighbourhood $U=U(h,c)$ of infinity and a quasiconformal map
$\psi = \psi(h,c)$ such that
\begin{equation}
\label{thm1eq1}
h(\psi(z))^2 = \psi(f(z)),
\end{equation}
for $z \in U$, where $f(z) = h(z)^2+c$. Further, $\psi$ is asymptotically conformal as $\av z \av \to \infty$.
\end{theorem}

\begin{remark}
Recall that $\psi$ is asymptotically conformal as $\av z \av \to \infty$ if for all $\epsilon >0$, there exists a neighbourhood $V$ of infinity
such that the complex dilatation of $\psi$ satisfies $\av \av \mu_{\psi}(z) \av \av < \epsilon$ for $z \in V$.
\end{remark}

\begin{remark}
Theorem \ref{thm1} also holds for $p(h(z))$, where $p$ is any polynomial of degree $d \geq 2$ and $h$ is affine. For simplicity, we restrict to the
case $p$ is a quadratic and recall from \cite{FG} that any composition of a quadratic and an affine mapping is linearly conjugate to a composition of a
quadratic of the form $z^2+c$ and an affine mapping.
\end{remark}

Recall the escaping set $I(f) = \{z \in \C: f^n(z) \to \infty \}$.
The quasiconformal map $\psi$ constructed in Theorem \ref{thm1} is initially defined in a neighbourhood of infinity, but we may extend its domain
of definition. We write $H(z)=h(z)^2$.

\begin{theorem}
\label{thm2}
\begin{enumerate}
\item If $0 \notin I(H+c)$, then $\psi$ can be continued injectively to a locally quasiconformal map $I(H+c) \to I(H)$.
\item If $0 \in I(H+c)$, then $\psi$ cannot be extended to the whole of $I(H+c)$, but may be extended injectively to a domain containing $c$.
\end{enumerate}
\end{theorem}

\begin{remark}
In case (i) of Theorem \ref{thm2}, we can only assert local quasiconformality.
The map $\psi$ is extended by pulling back under \eqref{thm1eq1}, and each time we pull back the distortion will increase.
Therefore the distortion will be unbounded as one approaches $\partial I(H+c)$.
\end{remark}

A uniformly quasiregular mapping is one for which there is a common bound on the distortion of all the iterates.
We can use Theorem \ref{thm1} to prove the following result on the mapping $h(z)^2+c$.

\begin{theorem}
\label{thm3}
Let $h$ be affine and $c \in \C$. Then the mapping $f(z) = h(z)^2+c$ is not uniformly quasiregular.
\end{theorem}

The significance of Theorem \ref{thm3} is as follows. By a result of Hinkkanen \cite{Hinkkanen}, every uniformly quasiregular
mapping in the plane is a quasiconformal conjugate of an analytic mapping. This is a generalization of results of Sullivan \cite{Sullivan} and Tukia \cite{Tukia} for uniformly quasiconformal mappings. The upshot of this is that the study of uniformly quasiregular mappings
in the plane reduces to the standard theory of complex dynamics. Therefore, for the study of the dynamics of mappings of the form
$h(z)^2+c$ to be of independent interest, we need to know that they are not uniformly quasiregular.

In view of Theorem \ref{thm1}, the proof of Theorem \ref{thm3} reduces to the following result.

\begin{theorem}
\label{thm3a}
Let $h$ be an affine mapping. Then $h^2$ is not uniformly quasiregular.
\end{theorem}

This theorem will be proved by showing that the complex dilatation of the iterates of $h^2$ on a ray fixed by $h^2$ has a particularly nice form.
Using this, and some basic iteration theory of M\"{o}bius transformations, we show that the modulus of the complex dilatation converges to $1$ on this fixed ray, which is equivalent to the maximal dilatation being unbounded.
Assuming this result for the moment, the proof of Theorem \ref{thm3} runs as follows.

\begin{proof}[Proof of Theorem \ref{thm3}]
Write $H(z)=h(z)^2$ and $f(z) = h(z)^2+c$. By Theorem \ref{thm3a}, $H$ is not uniformly quasiregular in any neighbourhood of infinity.
By Theorem \ref{thm1}, $H = \psi \circ f \circ \psi^{-1}$ in a neighbourhood of infinity $U$. Therefore
\[ K(H^n) = K(\psi \circ f^n \circ \psi^{-1}) \leq K(\psi)^2 K(f^n),\]
where $K(g)$ denotes the maximal dilatation of $g$.
Since $K(H^n) \to \infty$ in $U$, we have $K(f^n) \to \infty$ in $U$.
\end{proof}

The paper is organized as follows. In \S 3, we recall some preliminary material on quasiregular mappings and results from \cite{FG}. In \S 4, we outline the proof of Theorem \ref{thm1}, saving the details for \S 5. In \S 6 we prove Theorem \ref{thm2} and in \S 7 we prove Theorem \ref{thm3a}.

\section{Preliminaries}

\subsection{Quasiregular mappings}

We first collect some definitions and results that we will use. The standard reference for quasiregular mappings is Rickman's monograph \cite{Rickman}.

A quasiregular mapping $f:G \rightarrow \R^{n}$ from a domain $G \subseteq \R^{n}$ is called quasiregular if $f$ belongs to the Sobolev space $W^{1}_{n, loc}(G)$ and there exists $K \in [1, \infty)$ such that
\begin{equation}
\label{eq2.1}
\av f'(x) \av ^{n} \leq K J_{f}(x)
\end{equation}
almost everywhere in $G$. Here $J_{f}(x)$ denotes the Jacobian determinant of $f$ at $x \in G$. The smallest constant $K \geq 1$ for which (\ref{eq2.1}) holds is called the outer distortion $K_{O}(f)$. If $f$ is quasiregular, then we also have
\begin{equation}
\label{eq2.2}
J_{f}(x) \leq K' \inf _{\av h \av =1} \av f'(x) h \av ^{n}
\end{equation}
almost everywhere in $G$ for some $K' \in[1, \infty)$. The smallest constant $K' \geq 1$ for which (\ref{eq2.2}) holds is called the inner distortion $K_{I}(f)$. The maximal distortion $K=K(f)$ of $f$ is the larger of $K_{O}(f)$ and $K_{I}(f)$, and we then say that $f$ is $K$-quasiregular. In dimension $2$, we have $K_{O}(f) = K_{I}(f)$. An injective quasiregular mapping is called quasiconformal.

The degree of a mapping is the maximal number of pre-images and is in direct analogue with the degree of a polynomial. A quasiregular mapping is said to be of polynomial type if its degree is uniformly bounded at every point, or equivalently, if $\av f(x) \av \rightarrow \infty$ as $\av x \av \rightarrow \infty$.

Denote by $B(f)$ the branch set of $f$, that is, the set where $f$ is not locally injective.
A quasiconformal mapping is an injective quasiregular mapping. The following result says that in dimension $2$, a quasiregular mapping can be factorized into two mappings, one of which deals with the distortion and one which deals with the branch points.

\begin{theorem}[Stoilow factorization, see for example \cite{IM} p.254]
\label{Stoilow}
Let $f:\C \rightarrow \C$ be a quasiregular mapping. Then there exists an analytic function $g$ and a quasiconformal mapping $h$ such that $f = g \circ h$.
\end{theorem}

Stoilow factorization tells us what the branch set of a quasiregular mapping in $\C$ can be.

\begin{corollary}
Let $f:\C \rightarrow \C$ be quasiregular. Then $B(f)$ is a discrete set of points. In particular, if $f$ is quasiregular of polynomial type,
then $B(f)$ is a finite set of points.
\end{corollary}

In dimension $2$, the complex dilatation of a quasiconformal map $f$ is
\[ \mu_f = f_{\zbar}/f_z.\]
This is related to the distortion via
\[ K(f) = \frac{1+\av \av \mu_f \av \av _{\infty}}{1-\av \av \mu_f \av \av _{\infty} }.\]

\subsection{Dynamics of $h(z)^2+c$}

The type of quasiregular mappings of interest in this paper were first studied in \cite{FG}. We summarize the results of that paper in
this subsection.

The mappings to be iterated are quasiregular mappings of polynomial type in $\C$, of degree $2$, and with constant complex dilatation.
By Stoilow factorization, it follows that such
a mapping $f$ can be decomposed into an analytic mapping $g$
and a quasiconformal mapping $h$ such that $f=g \circ h$.
Since $f$ has degree $2$, $g$ must be a quadratic polynomial.

Consider an affine mapping $h=h_{K,\theta}:\C \rightarrow \C$ which stretches by a factor $K >0$ in the direction $e^{i\theta}$. If $\theta =0$,
then
\[ h_{K,0}(x+iy) = Kx+iy.\]
For general $\theta$, pre-compose $h_{K,0}$ by a rotation of $-\theta$ and post-compose by a rotation of $\theta$ to give the expression
\begin{equation*}
h_{K,\theta}(x+iy) = x(K\cos ^{2} \theta + \sin ^{2} \theta) + y(K-1)\sin \theta \cos \theta
\end{equation*}
\begin{equation}
\label{eq3.4}
+ i \left [ x(K-1)\cos \theta \sin \theta + y (K \sin ^{2} \theta + \cos ^{2} \theta ) \right ]
\end{equation}
or
\begin{equation}
\label{eq3.4a}
h_{K,\theta}(z) = \left ( \frac{ K+1}{2} \right ) z + e^{2i \theta}\left ( \frac{ K-1}{2} \right )\zbar.
\end{equation}
Using the formula for complex dilatation (see \cite{FM}), we see that
\begin{equation}
\label{eq3.3}
\mu _{h_{K,\theta}} = e^{2i\theta} \frac{K-1}{K+1},
\end{equation}
and so $\av \av \mu_{h_{K,\theta}} \av \av _{\infty} <1$ which means that $h_{K,\theta}$ is quasiconformal with constant complex dilatation.

\begin{proposition}{\cite[Proposition 3.1]{FG}}
\label{canonicalform}
Let $f:\C \rightarrow \C $ be a composition of a quadratic polynomial and an affine stretch of the form (\ref{eq3.4}). Then $f$ is linearly conjugate to
\begin{equation}
\label{eq3.8}
f_{K,\theta,c} := (h_{K,\theta})^{2}+c
\end{equation}
for some $Ke^{i\theta} \in \C \setminus \{ 0 \}$ and $c \in \C$.
Moreover, if we insist that $Ke^{i\theta} \in \Omega$, where
\begin{equation}
\label{eq3.20}
\Omega = \{\av z \av > 1, -\pi / 2 < \arg (z) \leq \pi / 2\} \cup \{ 1 \} ,
\end{equation}
then such a representation is unique.
\end{proposition}

\begin{definition}
If $f=f_{K,\theta,c}$ for $c \in \C$ and $Ke^{i\theta} \in \Omega$,
we say that $f \in QA$.
\end{definition}

\begin{remark}
This is a slightly different normalization for the set $\Omega$ compared to \cite{FG}.
\end{remark}

The properties of mappings in $QA$ studied in \cite{FG} are summarized in the following theorem.

\begin{theorem}{\cite[Theorems 4.3 and 4.5]{FG}}
\label{list}
Let $g$ be a polynomial of degree $d \geq 2$ and let $h$ be $L$-bi-Lipschitz. Let $f=g \circ h$, then $I(f)$ is a non-empty open set and $\partial I(f)$ is a perfect set.
Further, the family of iterates $\{f^{k}:k \in \N \}$ is equicontinuous on $I(f)$ and not equicontinuous at any point of $\partial I(f)$; the sets $I(f)$, $\partial I(f)$ and $\overline{I(f)} ^{c}$ are all completely invariant; and the escaping set is a connected neighbourhood of infinity.
\end{theorem}

\begin{definition}
Let $N(f)$ be the set of points whose orbits remain bounded
\begin{equation*}
N(f) = \{ z \in \C : \av f^{k} (z) \av <T , \text{ for some } T < \infty, \forall k \in \N \}.
\end{equation*}
\end{definition}

For $f \in QA$, it is easy to see that $N(f) = I(f)^{c}$. Further,  $N(f)$ is completely invariant by Theorem \ref{list}. This set is the direct analogue of the filled-in Julia set $K_{f}$ for polynomials, but here we are reserving the use of the symbol $K$ for distortion. Recall that $B(f)$ is the branch set of $f$,
and for $f \in QA$ the only point in the branch set is $0$.

\begin{proposition}{\cite[Theorems 5.3 and 5.4]{FG}}
\label{QAconn}
Let $f \in QA$. Then $N(f)$ is connected if and only if $I(f) \cap B(f) = \emptyset$. If $I(f)$ contains $B(f)$, then $N(f)$ has infinitely many components.
\end{proposition}

Contained in the proof of Theorem 6.4 of \cite{FG} is the observation that $h$ takes the form
\begin{equation}
\label{hpolar}
h(re^{i\varphi}) = r(1+(K^2-1)\cos^2(\varphi-\theta))^{1/2} \exp \left [ i \left( \theta + \tan^{-1} \left ( \frac{\tan(\varphi-\theta)}{K} \right ) \right ) \right ]
\end{equation}
in polar coordinates.
Hence every mapping of the form $h^2$ maps rays to rays and there exists a fixed ray, say of angle $\phi$,
where $\phi$ depends only on $K,\theta$. These fixed rays will play an important role in proving Theorem \ref{thm3a}, and will
also be studied in further detail in \cite{FF}.

\subsection{Logarithmic coordinates}

To prove Theorem \ref{thm1}, we will need to use the logarithmic transform which we briefly outline here.

Let $f$ be a function defined in a neighbourhood $U = \{ \av z \av >R \}$ of infinity and which grows like a polynomial.
That is, there exist constants $A,B,n$ such that
\[ A \leq \frac{ \av f(z) \av }{\av z \av ^n} \leq B.\]
Then $f$ lifts to a function
\[ \ft (X) = \log f(e^X) \]
for $\Re X > \log R$.

\begin{definition}
The function $\ft$ is called the logarithmic transform of $f$, and is unique upto addition of an integer multiple of $2\pi i$.
\end{definition}

\begin{lemma}
\label{logtranscomp}
Suppose $f,g$ are two functions whose logarithmic transforms exist. Then $\widetilde{f \circ g} = \ft \circ \widetilde{g}$ in a suitable neighbourhood of infinity.
\end{lemma}

\begin{proof}
This is obvious from the definition.
\end{proof}

\begin{lemma}
\label{glogtrans}
Let $g(z) = z^2+c$. Then $\widetilde{g}(X) = 2X+ \rho(X)$, where $\rho(X) = O(e^{-2\Re X})$.
\end{lemma}

\begin{proof}
We have
\begin{align*}
\widetilde{g}(X) &= \log (e^{2X}+c ) \\
&= \log ( e^{2X} (1+ce^{-2X})) \\
&= 2X + \log (1+ce^{-2X}),
\end{align*}
which proves the lemma.
\end{proof}

\begin{lemma}
Let $h= h_{K,\theta}$ be given by \eqref{eq3.4a}. Then
\begin{equation*}
\htt (X)=X + \log \left ( \frac{K+1}{2} + e^{2i\theta} \left(\frac{K-1}{2} \right) e^{-2i \Im X} \right ).
\end{equation*}
and
\begin{equation*}
\widetilde{h^{-1}}(X) = X +  \log \left ( \frac{K+1}{2K} - e^{2i\theta} \left(\frac{K-1}{2K} \right) e^{-2i \Im X} \right).
\end{equation*}
\end{lemma}

\begin{proof}
This is obvious from the definition of $h$.
\end{proof}

\begin{definition}
\label{varphixi}
We define the functions
\[ \varphi(X) = \htt(X) - X,\]
and
\[ \xi(X) = \widetilde{h^{-1}}(X) - X.\]
\end{definition}

It is clear from the definition that $\av \varphi\av , \av \xi\av $ are both bounded above and below.

Recalling that $f = g \circ h$, it follows that, using the notation above, the logarithmic transform of $f$ is
\begin{equation}
\label{flogtrans}
\ft(X) = 2X + 2\varphi(X) + \rho( X+\varphi(X)).
\end{equation}
To see that $\ft$ is well-defined, note that
\[ \ft(X+2\pi i) = 2X +4 \pi i +2 \varphi( X+2 \pi i) + \rho(X+2\pi i + \varphi(X+2\pi i) ).\]
It is easy to see that $\varphi (X+2 \pi i) = \varphi (X)$, and so
\[ \ft(X+2\pi i) - \ft(X) = 4 \pi i+ \rho(X+2\pi i + \varphi(X+2\pi i) ) - \rho(X + \varphi (X)).\]
The left hand side of this equation is a multiple of $2 \pi i$, whereas the right hand side
differs from a multiple of $2 \pi i$ by something small for large $\Re X$, and hence by $0$.

\section{Proof of Theorem \ref{thm1}}

\subsection{Outline}

Let $g(z) = z^2+c$ and $h=h_{K,\theta}$ for $K>1$ and $\theta \in (-\pi / 2, \pi /2]$.
Then $f=g \circ h\in QA$ and we consider $f$ in a neighbourhood of infinity, say $U=\{ \av z \av >R\}$.
To prove Theorem \ref{thm1}, we will do the following.

\begin{itemize}
\item Writing $H=h^2$, define a branch $\psi_1$ of $H^{-1} \circ f$ in $U$.
\item Show $\psi_1 (z) = z +o(1)$ near infinity and $\psi_1$ is asymptotically conformal.
\item Inductively define a branch $\psi_{k+1}$ of $H^{-(k+1)} \circ f^{k+1}$ in $U$ by considering
$H^{-1} \circ \psi_k \circ f$.
\item Show $\psi_k(z) = z +o(1)$ near infinity and $\psi_k$ is asymptotically conformal.
\item Show the sequence $\psi_k$ converges locally uniformly to the required B\"{o}ttcher coordinate.
\end{itemize}

\subsection{The sequence $\psi_k$}

Firstly, define an analytic branch $p_{1}$ of $\log (1 + c/z^{2})$ in $U$, shrinking $U$ if necessary,
so that $\lim _{z \rightarrow \infty} \varphi _{1}(z) = 0$.
Then $g(z) = z^{2} \exp p_{1}(z)$ in $U$ and we can choose an analytic square root $q_{1}$ given by
\begin{equation*}
\label{conj3}
q_{1}(z) = z \exp p_{1} (z)/2.
\end{equation*}
such that $q_{1}^{2} =g$ in $U$.
We can also assume that $q_{1}$ is injective in $U$,
since if $q_{1}(z) = q_{1}(w)$, then $g(z) = g(w)$ and so $z=\pm w$, but $q_{1}(w) \neq q_{1}(-w)$ since expanding the expression for $q_1$ gives $q_{1}(z) = z+o(1)$ near infinity.
Then we define
\begin{equation*}
\psi _1(z) = h^{-1}(q_1(h(z))).
\end{equation*}
We can write $\psi _1(z) = z+R_1(z)$, and assume for now that $R_1(z) =o(1)$ for large $\av z \av$.

We continue defining the functions $\psi_k(z)=z+R_k(z)$ by induction. For $k \geq 1$, define a continuous branch $p_{k}$ of
\begin{equation*}
\log \left ( 1 + \frac{c+R_{k-1}(z^{2}+c)}{z^{2}} \right )
\end{equation*}
in $U$ so that $\lim _{z \rightarrow \infty} p _{k}(z) = 0$,
assuming $R_{k-1}(z)=o(1)$.
Then $\psi_{k-1} (g(z)) = z^{2} \exp p_{k}(z)$ in $U$ and we can choose a continuous square root $q_{k}=z\exp(p_k/2)$ such that $q_{k}^{2} =\psi_{k-1} \circ g$ in $U$.
We also observe that $q_{k}$ is injective near infinity, since if $q_{k}(z) = q_{k}(w)$, then $\psi_{k-1}(g(z)) = \psi_{k-1}(g(w))$ and so $z=\pm w$ since $\psi_{k-1}$ is injective, but $q_{k}(z) \neq q_{k}(-w)$ since expanding the expression for $q_k$ gives $q_{k}(z) = z+o(1)$ in $U$.
This means that $\psi _k= h^{-1} \circ q_k \circ h$ is injective in a neighbourhood of infinity.

To prove Theorem \ref{thm1}, we will need to prove the following propositions.

\begin{proposition}
\label{bprop1}
The functions $\psi_k$ can be written as
\[\psi_k(z) = z+ R_k(z),\]
in $U$, where $R_k(z) = o(1)$.
Moreover, the $\psi_k$ converge uniformly to a function $\psi$
in $U$ and
\[ \psi (z) = z + R(z),\]
where $R(z) = o(1)$.
\end{proposition}

\begin{proposition}
\label{bprop2}
The function $\psi$ is quasiconformal in $U$ and, further, is asymptotically conformal.
\end{proposition}

We will postpone the proof of these two propositions until the next section. It seems difficult to prove these propositions directly, and so the proofs make use of the logarithmic transforms of the functions
$\psi_k$.

With these results in hand,
by the construction,
\begin{equation*}
h(z)^{2} = \psi_{k-1}(f(\psi_{k}^{-1}(z)))
\end{equation*}
for all $k \geq 1$.
Taking the limit as $k \to \infty$, we have $\psi (f( \psi ^{-1}(z))) = h(z)^{2}$ for $z \in U$.
That is, the following diagram commutes.
\begin{center}
\begin{tikzpicture}[description/.style={fill=white,inner sep=2pt}]
\matrix (m) [matrix of math nodes, row sep=3em,
column sep=2.5em, text height=1.5ex, text depth=0.25ex]
{U & \psi(U) \\
f(U) & h^2(\psi(U)) \\ };
\path[->,font=\scriptsize]
(m-1-1) edge node[auto] {$ \psi $} (m-1-2)
(m-1-2) edge node[auto] {$ h^2 $} (m-2-2)
(m-1-1) edge node[auto] {$ f $} (m-2-1)
(m-2-1) edge node[auto] {$ \psi $} (m-2-2);
\end{tikzpicture}
\end{center}
This proves the theorem.

\section{Logarithmic transforms of $\psi_k$}

In this section, we will take the logarithmic transforms of the $\psi_k$ and use them to prove Propositions \ref{bprop1} and \ref{bprop2}.
Let $L$ be the half-plane $\Re (X) >\sigma$, where $\sigma$ is large, and so $L$ corresponds to a neighbourhood $U$ of infinity in the $z$-plane.
In $L$,
for $k \geq 0$, define $F_0(X) = X$ and
\begin{equation}
\label{fk+1}
F_{k+1} = \widetilde{h^{-1}} \circ \widetilde{S} \circ F_k \circ \widetilde{g} \circ \htt ,
\end{equation}
where $\widetilde{S}(X)=X/2$,
and write
\begin{equation*}
F_k(X) = X+ T_k(X).
\end{equation*}
Here, $T_k$ measures how far away $F_k$ is from the identity in $L$.
Then the logarithmic transform of our sequence $\psi_k$ is $\widetilde{\psi_k}(X) = F_k(X)$ by Lemma \ref{logtranscomp}.

\subsection{Preliminary observations}

We first fix $\alpha \in (1,2)$. The role that $\alpha$ plays will be seen in Lemmas \ref{blemma3} and \ref{blemma4}.
We will work with $X \in L = \{ Z: \Re Z>\sigma\}$ where $\sigma$
may be larger than $\log R$, and will depend on $K,\theta,c,\alpha$.
The constants $C_j$ which appear will all depend on at least $K,\theta,c$, and may have other dependencies, which will be stated.

Recall $\varphi, \xi$ from Definition \ref{varphixi}.

\begin{lemma}
\label{blemma0}
There exists a constant $C_1>0$ such that $|\varphi(X)|<C_1$ and $|\xi(X)|<C_1$ for all $X \in L$. Further, we have \[\varphi (X) + \xi ( X + \varphi (X) ) = 0\] and
\[ \xi (X) + \varphi ( X+ \xi (X)) = 0.\]
\end{lemma}

\begin{proof}
The first part follows from the definition of $\varphi$ and $\xi$ since $e^{2i\theta}(K-1)/(K+1) \in \D$.
The second part is just translating the fact that $h$ and $h^{-1}$ are mutual inverses to the logarithmic coordinate setting.
\end{proof}

The following corollary follows by differentiating the identities of Lemma \ref{blemma0}.

\begin{corollary}
\label{blemma0cor}
The partial derivatives of $\varphi$ and $\xi $ satisfy
\begin{equation*}
\varphi_X(X) +  \xi _X (X+\varphi(X))(1+\varphi_X(X)) + \xi_{\xbar} (X+\varphi(X)) \overline{ \varphi_{\xbar}(X)} = 0
\end{equation*}
and
\begin{equation*}
\varphi_{\xbar}(X) + \xi_X(X+\varphi(X)) \varphi_{\xbar}(X) + \xi _{\xbar}(X+\varphi(X))(\overline{1+\varphi_X(X)}) = 0.
\end{equation*}
\end{corollary}

Next, we consider small variations of $\varphi$ and $\xi$.

\begin{lemma}
\label{blemma1}
Given $\delta>0$, there exists $C_2>0$ depending on $\delta$ such that for $\av Y\av<\delta$,
we have \[\av \varphi (X+Y) - \varphi (X) \av < C_2\av Y \av \] for any $X \in L$,
and the same holds for $\xi$.
\end{lemma}

\begin{proof}
Write
\[ \nu = e^{2i\theta}\left ( \frac{K-1}{K+1} \right),\]
with $K \geq 1$,
so that $\nu \in \D$. Then, expanding $e^{-2i \Im Y}$ shows that
\begin{align*}
\av \varphi(X+Y) - \varphi (X) \av &= \left \av \log \frac{ 1 + \nu e^{-2i \Im (X+Y)}}{1+\nu e^{-2i\Im(X)}} \right \av \\
&= \left \av \log  \left ( 1 - \left ( \frac{2i\nu e^{-2i\Im(X)} }{1+\nu e^{-2i\Im(X)}} \right ) \Im (Y) + O((\Im Y)^2) \right ) \right \av \\
&\leq \left \av \frac{2i\nu e^{-2i\Im(X)} }{1+\nu e^{-2i\Im(X)}} \right \av  \av \Im Y \av + o(\av \Im Y \av).
\end{align*}
Since $|\Im Y| \leq |Y|$ and the coefficient of $\av \Im Y \av$ in the latter expression is uniformly bounded because $\nu \in \D$, we have the required conclusion.
Analogous calculations hold for $\xi$.
\end{proof}

The following lemma is the analogue of Lemma \ref{blemma1} for the partial derivatives.

\begin{lemma}
\label{blemma1a}
Given $\delta >0$, there exists $C_3>0$ depending on $\delta$ such that for all $\av Y\av<\delta$,
we have \[\av \varphi_X (X+Y) - \varphi_X (X) \av < C_3\av Y \av \] for any $X \in L$,
and the same holds for $\varphi_{\xbar},\xi_X$ and $\xi_{\xbar}$.
Further, there exists $C_4$ such that the modulus each of the partial derivatives
is uniformly bounded above, i.e. $|\varphi_X(X)|<C_4$ for $X \in L$ etc.
\end{lemma}

\begin{proof}
We note that the partial derivatives of $\varphi$ and $\xi$ are
\[ \varphi_X(X) = -\frac{\nu e^{-2i\Im(X)}}{1+\nu e^{-2i\Im(X)}}, \:\:\:\:\: \varphi_{\xbar}(X) = \frac{\nu e^{-2i\Im(X)}}{1+\nu e^{-2i\Im(X)}}\]
and
\[ \xi_X(X) = \frac{\nu e^{-2i\Im(X)}}{1-\nu e^{-2i\Im(X)}}, \:\:\:\:\: \xi_{\xbar}(X) = -\frac{\nu e^{-2i\Im(X)}}{1+\nu e^{-2i\Im(X)}}.\]
Then, we have
\begin{align*}
\av \varphi_X (X+Y) - \varphi_X (X) \av &=
\left \av -\frac{\nu e^{-2i\Im(X+Y)}}{1+\nu e^{-2i\Im(X+Y)}} + \frac{\nu e^{-2i\Im(X)}}{1+\nu e^{-2i\Im(X)}} \right \av\\
&= \left \av  \frac{ \nu e^{-2i\Im(X)}(1-e^{-2i\Im(Y)})}{(1+\nu e^{-2i\Im(X)})(1+\nu e^{-2i\Im(X+Y)})} \right \av \\
&\leq \left \av \frac{ 2i\nu e^{-2i\Im(X)}}{(1+\nu e^{-2i\Im(X)})(1+\nu e^{-2i\Im(X+Y)})} \right \av \av \Im Y \av +o(\av \Im Y \av).
\end{align*}
The denominator in the coefficient of $\av \Im Y\av$ is uniformly bounded since $\nu \in \D$, and since $|\Im Y| \leq |Y|$, we have the desired conclusion.
The calculations for the other partial derivatives run analogously.
The final part of the lemma follows since $\nu \in \D$.
\end{proof}

We may assume that $\sigma$ is chosen so large that there exists $C_5>0$ such that
\begin{equation}
\label{rdecay}
\left \av \rho(X+\varphi(X)) \right \av < C_5e^{-2\Re X}
\end{equation}
for all $X \in L$.
Next, consider the behaviour of $\ft$ for $X \in L$, recalling \eqref{flogtrans}.

\begin{lemma}
\label{fbehav}
There exists a constant $C_6>0$ such that
\begin{equation*}
|\Re \ft (X) - 2 \Re X| < C_6,
\end{equation*}
for $X \in L$.
\end{lemma}

\begin{proof}
Recall the definition of $\ft$ from \eqref{flogtrans}. Then
\[ |\Re \ft (X) - 2 \Re X| \leq 2|\varphi(X)| + | \rho(X+\varphi(X))|.\]
By Lemma \ref{blemma0} and \eqref{rdecay}, this gives
\[ |\Re \ft (X) - 2 \Re X| <2C_1 + C_5e^{-2\Re X},\]
which proves the lemma.
\end{proof}

We note that in applications of Lemma \ref{fbehav}, we will usually use the inequality
\[ \Re \ft (X) >2\Re X -C_6,\]
for $X \in L$.

\subsection{Growth of $F_k$}
In this section, we will estimate how $|F_k|$ grows for large $\Re X$, and also show the the difference between successive terms in the sequence gets smaller as $k$ increases.

First, recall that $F_1 = \widetilde{h^{-1}} \circ \widetilde{S} \circ \widetilde{g} \circ \htt$. Writing this out in full gives
\begin{align}
\label{f1def}
F_1(X) &= X+\varphi(X) + \frac{\rho(X+\varphi (X))}{2} + \xi \left ( X+\varphi(X) + \frac{\rho(X+\varphi(X))}{2}  \right ).
\end{align}
Recall also that $T_k(X) = F_k(X)-X$ is the function that shows how far $F_k$ deviates from the identity.

\begin{lemma}
\label{blemma2}
There exists a constant $C_7>0$ such that
\[\av T_1(X) \av \leq C_7e^{- 2\Re X },\]
for $X \in L$.
\end{lemma}

\begin{proof}
Applying Lemma \ref{blemma1} with $Y = \frac{\rho(X+\varphi(X))}{2}$ shows that
\[ \left \av  \xi \left ( X+\varphi(X) + \frac{\rho(X+\varphi(X))}{2} \right ) -  \xi \left ( X+\varphi(X)  \right ) \right \av <C_2
 \left\av \frac{\rho(X+\varphi(X))}{2} \right\av.\]
Recall from Lemma \ref{blemma0} that $\varphi(X) +\xi(X+\varphi(X))=0$. Then from \eqref{f1def} we obtain that
\[ \left \av T_1(X) \right \av < (1+C_2) \left\av \frac{\rho(X+\varphi(X))}{2} \right\av.\]
Finally, using \eqref{rdecay} implies the lemma.
\end{proof}

Recall that $\alpha \in(1,2)$. The reason $\alpha $ is introduced is the following lemma. Namely, the fact that $\alpha$ is less than $2$ allows us to give an estimate on the growth of the $T_k$ which is valid for all $k$.

\begin{lemma}
\label{blemma3}
There exists a constant $C_8>0$ depending on $\alpha$ such that
for all $k \geq 1$, we have
\[ \av T_k(X) \av <C_8 e^{-\alpha\Re X},\]
for $X \in L$.
\end{lemma}

\begin{proof}
We will proceed by induction. By Lemma \ref{blemma2}, the result is true for $k=1$ if $C_8>C_7e^{(\alpha-2)\sigma}$, recalling that $\Re X > \sigma$. Let us assume then that
\begin{equation}
\label{blemma3eq1}
\av T_k(X) \av <C_8 e^{-\alpha\Re X}.
\end{equation}
We may assume that $\sigma$ is large enough that \eqref{rdecay}
is satisfied and we may apply Lemma \ref{blemma1}
with $Y =\rho (X+\varphi(X))/2 + T_k(\ft(X))/2$, so that
\begin{equation}
\label{blemma3eq2}
\left \av \xi \left(X+\varphi(X) + \frac{\rho(X+\varphi(X))}{2} + \frac{ T_k(\ft(X))}{2} \right) -
\xi (X+\varphi(X)) \right \av < C_2 \left \av \frac{\rho(X+\varphi(X))}{2} + \frac{ T_k(\ft(X))}{2} \right \av,
\end{equation}
for $X \in L$.
Using \eqref{fk+1}, we can write $F_{k+1}$ as
\begin{align}
\label{blemma3eq3}
F_{k+1}(X) &= X + \varphi (X) + \frac{\rho(X+\varphi(X))}{2} + \frac{ T_k ( \ft(X) ) }{2}\\
\notag &+ \xi \left (X + \varphi (X) + \frac{\rho(X+\varphi(X))}{2} + \frac{ T_k ( \ft(X))}{2} \right ).
\end{align}
Recalling from Lemma \ref{blemma0} that $\varphi(X)+\xi(X+\varphi(X))=0$,
then \eqref{blemma3eq3} implies that
\[ \left | T_{k+1}(X)\right | < \left ( \frac{1+C_2}{2} \right ) \left | \rho(X+\varphi(X)) + T_k(\ft(X)) \right |.\]
By the inductive hypothesis and Lemma \ref{fbehav},
\begin{align*}
|T_k(\ft(X))| &<C_8e^{-\alpha \Re \ft(X)} \\
&< C_8e^{\alpha C_6}e^{-2\alpha \Re X}.
\end{align*}
Using this and \eqref{rdecay}, we obtain
\begin{align*}
\left | T_{k+1}(X)\right | &< \left ( \frac{1+C_2}{2} \right ) \left( C_5e^{-2\Re X} + C_8e^{\alpha C_6}e^{-2\alpha \Re X} \right ) \\
&= e^{-\alpha \Re X} \left ( \frac{(1+C_2)C_5}{2} \: e^{(\alpha-2)\Re X} + \frac{(1+C_2)C_8}{2}\: e^{\alpha C_6}e^{-\alpha \Re X} \right ) \\
&< e^{-\alpha \Re X} \left ( \frac{(1+C_2)C_5}{2} \: e^{(\alpha-2)\sigma} + \frac{(1+C_2)C_8}{2}\: e^{\alpha C_6}e^{-\alpha \sigma} \right ).
\end{align*}
We may assume that $\sigma$ was chosen so large that
\[ (1+C_2)\: e^{\alpha C_6}e^{-\alpha \sigma} <1,\]
and also $C_8$ is large enough that
\[ (1+C_2)C_5 \: e^{(\alpha-2)\sigma} < C_8,\]
from which it follows that
\[ \left | T_{k+1}(X)\right |  < C_8e^{-\alpha \Re X},\]
which proves the lemma.
\end{proof}

\begin{lemma}
\label{blemma4}
For all $k \geq 1$, there exists a constant $C_9$ depending on $\alpha$ such that
\begin{equation*}
\left \av F_{k+1}(X) - F_{k}(X) \right \av < C_9 e^{-\alpha^k \Re X},
\end{equation*}
for $X \in L$.
\end{lemma}

\begin{proof}
Recalling that $F_0(X)=X$,
the lemma holds for $k=0$ by Lemma \ref{blemma2}. We proceed by induction, and assume that for some $k\geq 1$, we have
\begin{equation*}
\left \av F_{k}(X) - F_{k-1}(X) \right \av < C_9 e^{-\alpha^{k-1} \Re X},
\end{equation*}
noting that this is equivalent to
\begin{equation}
\label{blemma4eq1}
\left \av T_{k}(X) - T_{k-1}(X) \right \av < C_9 e^{-\alpha^{k-1} \Re X}.
\end{equation}
Using \eqref{blemma3eq3} applied to $F_{k+1}$ and $F_k$, we have that
\begin{align*}
F_{k+1}(X) - F_k(X) &= \xi\left (X+ \varphi(X) + \frac{\rho(X+\varphi(X))}{2} + \frac{T_k(\ft(X))}{2} \right )\\
&- \xi\left (X+ \varphi(X) + \frac{\rho(X+\varphi(X))}{2} + \frac{T_{k-1}(\ft(X))}{2} \right )\\
&+ \frac{T_k(\ft(X))}{2} - \frac{T_{k-1}(\ft(X))}{2}.
\end{align*}
Using Lemma \ref{blemma1} applied to $\xi$ with
\[ Y = \frac{T_k(\ft(X))}{2} - \frac{T_{k-1}(\ft(X))}{2},\]
we see that
\begin{equation}
\label{blemma4eq2}
\av F_{k+1}(X) - F_{k}(X) \av \leq \left \av (1+C_2) \left (
\frac{T_k(\ft(X))}{2} - \frac{T_{k-1}(\ft(X) )}{2} \right ) \right \av.
\end{equation}
The inductive hypothesis and Lemma \ref{fbehav} imply that
\begin{align*}
| T_{k}(\ft(X)) - T_{k-1}(\ft(X)) | &< C_9e^{-\alpha^{k-1}\Re \ft(X)} \\
&< C_9e^{C_6\alpha^{k-1}}e^{-2\alpha^{k-1} \Re X} \\
&< C_9e^{\alpha^{k-1}(C_6 - (2-\alpha) \sigma )}e^{-\alpha^k \Re X},
\end{align*}
for $X \in L$. Hence if $\sigma $ is chosen large enough that
$e^{\alpha^{k-1}(C_6 - (2-\alpha) \sigma )} < 2(1+C_2)^{-1}$ for $k\geq 1$, then we obtain from \eqref{blemma4eq2} that
\[ \av F_{k+1}(X) - F_{k}(X) \av  < C_9 e^{-\alpha^k \Re X},\]
which proves the lemma.
\end{proof}

\subsection{Complex dilatation of $F_k$}

In this section, we will estimate the growth of the complex dilatation of $F_k$ for large $\Re X$.
We will use the following formula for the complex derivatives of a
composition repeatedly, see for example \cite{FM}.

\begin{lemma}
\label{complexderivs}
The complex derivatives of compositions are
\[ (g \circ f)_z = (g_z \circ f)f_z + (g_{\zbar} \circ f) \overline{f_{\zbar}},\]
and
\[ (g \circ f)_{\zbar} = (g_z \circ f)f_{\zbar} + (g_{\zbar} \circ f)\overline{f_z}.\]
\end{lemma}

As a first application of this, we consider the complex derivatives of $\rho(X + \varphi(X))$.

\begin{lemma}
\label{rhocomplexderivs}
Let $\rho_1(X) = \rho(X+\varphi(X))$. Then there exists a constant $C_{10}>0$ such that
\[ | (\rho_1)_X(X) | \leq C_{10}e^{-2X}\text{ and } | (\rho_1)_{\xbar}(X)| \leq C_{10}e^{-2X},\]
for $X \in L$.
\end{lemma}

\begin{proof}
Recall from Lemma \ref{glogtrans} that $\rho(X) = \log (1+ce^{-2X})$.
Since $\rho$ is analytic, it follows that $\rho_{\xbar} \equiv 0$, and also
\[ \rho_X(X) = \frac{-2ce^{-2X}}{1+ce^{-2X}}.\]
Then using Lemma \ref{complexderivs} and Lemma \ref{blemma1a}, we have
\begin{align*}
| (\rho_1)_X(X)| &\leq | \rho_X(X+\varphi(X)) \cdot (1+\varphi_X(X)) + \rho_{\xbar}(X+\varphi(X)) \cdot \overline{ \varphi _{\xbar}(X)} |\\
&\leq (1+C_4) | \rho_X(X+\varphi(X)) |,
\end{align*}
which gives the desired conclusion for $(\rho_1)_X$, since $|\varphi|$ is bounded above by Lemma \ref{blemma0}. Similar calculations give the growth for $(\rho_1)_{\xbar}$.
\end{proof}

We now want to estimate the complex dilatations
$\mu_k$ of $F_k$.

\begin{proposition}
\label{blemma8}
There exist constants $C_{11},C_{12}>0$ such that for all $k \geq 1$,
\[ \av (F_k)_X(X) \av \geq 1 - C_{11}e^{-\alpha\Re X}\]
and
\[ \av (F_k)_{\xbar}(X) \av \leq C_{12} e^{-\alpha \Re X} \]
for all $X \in L$.
\end{proposition}

The proof of this proposition will proceed by induction. Since $F_0(X)=X$,
it is clear that the proposition holds for $k=0$. Hence assume the result is true for $k$. Recalling that $F_k(X) = X+T_k(X)$, this means that
\begin{equation}
\label{blemma8eq10}
\av (T_k)_X(X) \av \leq C_{11}e^{-\alpha\Re X}, \:\:\:\:\:
\av (T_k)_{\xbar}(X) \av \leq C_{12}e^{-\alpha\Re X}.
\end{equation}

\begin{lemma}
\label{blemma8a}
There exists constants $C_{13},C_{14}>0$ such that
\[ \left | \left [ T_k(\ft(X)) \right ] _X \right | < C_{13}e^{-2 \Re X}\]
and
\[ \left | \left [ T_k(\ft(X)) \right ] _{\xbar} \right | < C_{14}e^{-2 \Re X},\]
for $X \in L$.
\end{lemma}

\begin{proof}
By the inductive hypothesis \eqref{blemma8eq10}, we have
\[ \av (T_k)_X(\ft(X)) \av \leq C_{11}e^{-\alpha\Re \ft(X)}.\]
Recalling the growth of $\ft$ from Lemma \ref{fbehav}, this gives
\begin{align*}
\av (T_k)_X(\ft(X)) \av &< C_{11} e^{\alpha C_6} e^{-2\alpha \Re X}\\
&< C_{11} e^{C_6 \alpha + 2(1-\alpha)\sigma}e^{-2\Re X},
\end{align*}
for $X \in L$, which gives the result for $\left [ T_k(\ft(X)) \right ] _X$. The result for $\left [ T_k(\ft(X)) \right ] _{\xbar} $ follows analogously.
\end{proof}

Recalling the definition of $F_{k+1}$ from \eqref{fk+1}, we have
\[ F_{k+1}(X) = \frac{ F_k(\ft(X))}{2} + \xi \left ( \frac{ F_k(\ft(X))}{2} \right ).\]
For convenience let us write
\begin{equation}
\label{peq}
P(X) = \frac{ F_k(\ft(X))}{2}= X + \varphi(X) + \frac{\rho(X+\varphi(X))}{2} + \frac{ T_k(\ft(X))}{2},
\end{equation}
so that
\begin{equation*}
F_{k+1}(X) = P(X) + \xi(P(X)).
\end{equation*}
The complex derivatives of $P$ are
\begin{equation}
\label{pcd1}
P_X(X) = 1+ \varphi_X(X) + \left [ \frac{\rho(X+\varphi(X))}{2} \right ] _X + \left [ \frac{ T_k(\ft(X))}{2} \right ]_{X},
\end{equation}
and
\begin{equation}
\label{pcd2}
P_{\xbar}(X) = \varphi_{\xbar}(X) + \left [ \frac{\rho(X+\varphi(X))}{2}  \right ] _{\xbar} + \left [ \frac{ T_k(\ft(X))}{2} \right ]_{\xbar}.
\end{equation}

We are now in a position to Prove Proposition \ref{blemma8}.

\begin{proof}[Proof of Proposition \ref{blemma8}]
The complex derivative of $F_{k+1}$ with respect to $X$ is
\begin{equation*}
(F_{k+1})_X(X) = P_X(X) + P_X(X) \xi_X(P(X)) + \overline{ P_{\xbar}(X)} \xi_{\xbar}(P(X)).
\end{equation*}
Using the identity from Corollary \ref{blemma0cor}, we can write
\begin{align*}
(F_{k+1})_X(X) - 1 &=
\left ( P_X(X) - 1 - \varphi_X(X) \right) \\
&+ \left ( P_X(X) \xi_X(P(X)) - (1+\varphi_X(X)) \xi_X(X+\varphi(X)) \right) \\
&+ \left ( \overline{ P_{\xbar}(X)} \xi_{\xbar} (P(X)) - \overline{\varphi_{\xbar}(X) } \xi_{\xbar}(X+\varphi(X)) \right )\\
&= I_1 + I_2 + I_3.
\end{align*}
For $I_1$, by \eqref{pcd1} we have
\begin{align*}
|I_1|= \left | P_X(X) - 1 - \varphi_X(X) \right | &= \left | \left [ \frac{\rho(X+\varphi(X))}{2} \right ] _X + \left [ \frac{ T_k(\ft(X))}{2} \right ]_{X} \right | \\
&< \frac{C_{10}}{2} e^{-2\Re X} + \frac{C_{13}}{2} e^{-2\Re X} \\
&= \frac{(C_{10}+C_{13})}{2} e^{-2\Re X}
\end{align*}
by Lemmas \ref{rhocomplexderivs} and \ref{blemma8a}.

For $I_2$, first observe that by \eqref{rdecay} and Lemma \ref{blemma3}, we may assume that $\sigma$ is large enough that $\av P(X) - X - \varphi(X)\av < \delta$ for $X \in L$,
and so Lemma \ref{blemma1a} shows that
\begin{equation*}
\av \xi_X (P(X)) - \xi_X (X+\varphi(X)) \av <
C_3 \av P(X) - (X+\varphi(X)) \av,
\end{equation*}
for $X \in L$. By the definition of $P$, \eqref{rdecay} and the proof of Lemma \ref{blemma3}, this implies that there exists $C_{15}>0$ such that
\begin{align}
\label{blemma8eq2}
\av \xi_X (P(X)) - \xi_X (X+\varphi(X)) \av &<C_3\left( \frac{C_5}{2}\: e^{-2\Re X} + \frac{C_8}{2}\: e^{-\alpha \Re \ft (X)} \right )\\
\notag &< C_{15}e^{-2\Re X},
\end{align}
for $X \in L$.
Next, by \eqref{pcd1}, Lemma \ref{blemma1a} and the calculation for $I_1$, we have
\begin{equation}
\label{blemma8eq3}
|P_X(X)| < 1 +C_4 + \left ( \frac{C_{10}+C_{13}}{2} \right ) e^{-2\Re X} < C_{16},
\end{equation}
for $X \in L$.
Then \eqref{blemma8eq2}, \eqref{blemma8eq3}, Lemma \ref{blemma1a} for $|\xi_X|$ and the calculation for $I_1$ give
\begin{align*}
\av I_2 \av &= \av P_X(X) \xi_X(P(X)) - (1+\varphi_X(X)) \xi_X(X+\varphi(X)) \av  \\
&= \av P_X(X) [ \xi_X(P(X)) - \xi_X ( X+\varphi (X)  ) ]+ \xi_X(X+\varphi(X) ) [ P_X(X) -(1+\varphi_X(X) ) ] \av \\
&< C_{16}C_{15}e^{-2 \Re X} +\frac{C_4(C_{10}+C_{13})}{2} e^{-2\Re X}
\end{align*}

For $I_3$, observe first that since we may assume $\sigma$ is large enough that $\av P(X) - X - \varphi(X)\av < \delta$ for $X \in L$, Lemma \ref{blemma1a} implies that
\begin{equation*}
\av \xi_{\xbar} (P(X)) - \xi_{\xbar} (X+\varphi(X)) \av <
C_3 \av P(X) - (X+\varphi(X)) \av.
\end{equation*}
As in the calculation for $I_2$,  this implies that there exists $C_{17}>0$ such that
\begin{equation}
\label{blemma8eq6}
\av \xi_{\xbar} (P(X)) - \xi_{\xbar} (X+\varphi(X)) \av < C_{17}e^{-2\Re X},
\end{equation}
for $X \in L$.
Also observe that by \eqref{pcd2}, Lemma \ref{blemma1a} and the calculation for $I_1$ that there exists $C_{18}>0$ such that
\begin{equation}
\label{blemma8eq6a}
\av P_{\xbar}(X) \av <C_4 + \left ( \frac{C_{10}+C_{13}}{2} \right) e^{-2\Re X}<C_{18},
\end{equation}
for $X \in L$. Further, \eqref{pcd2} and calculations analogous to those for $I_1$ show that there exists $C_{19}>0$ such that
\begin{equation}
\label{blemma8eq6b}
| P_{\xbar}(X) - \varphi_{\xbar}(X) |<C_{19}e^{-2\Re X}.
\end{equation}
Then \eqref{blemma8eq6}, \eqref{blemma8eq6a}, \eqref{blemma8eq6b} and Lemma \ref{blemma1a} for $|\xi_{\xbar}|$ give
\begin{align*}
\av I_3 \av &= \av \overline{P_{\xbar}(X)} \xi_{\xbar}(P(X)) - \overline{\varphi_{\xbar}(X)} \xi_{\xbar}(X+\varphi(X)) \av  \\
&= \av \overline{P_{\xbar}(X)} [ \xi_{\xbar}(P(X)) - \xi_{\xbar} ( X+\varphi (X)  ) ]+ \xi_{\xbar}(X+\varphi(X) ) [ \overline{P_{\xbar}(X)} -\overline{\varphi_{\xbar}(X)}  ] \av \\
&< C_{18}C_{17}e^{-2\Re X} + C_4C_{19}e^{-2\Re X},
\end{align*}
for $X \in L$.
The estimates for $I_1,I_2,I_3$ show that there exists $C_8'>0$ such that
\begin{equation*}
\av (F_{k+1})_X(X) - 1 \av < C_{8}'e^{-2 \Re X},
\end{equation*}
for $X \in L$ and hence if $\sigma$ is chosen large enough so that $C_8'e^{(\alpha-2)\sigma} <C_8$, then
\begin{equation*}
\av (F_{k+1})_X(X) - 1 \av < C_{8}e^{-\alpha \Re X},
\end{equation*}
Therefore
\begin{equation*}
\av (F_{k+1})_X(X)  \av >1-  C_{8}e^{-\alpha \Re X}
\end{equation*}
for $X \in L$ as required.

We next move on to estimate $\av(F_{k+1})_{\xbar}(X)\av$. The calculations are very similar to those above, but are included for the reader's convenience. From the definition of $F_{k+1}$ and Lemma \ref{complexderivs}, we have
\begin{equation*}
(F_{k+1})_{\xbar}(X) = P_{\xbar}(X) + \xi_{X}(P(X))P_{\xbar}(X) + \xi_{\xbar}(P(X)) \overline{P_X(X)}.
\end{equation*}
Using the second identity from Corollary \ref{blemma0cor}, we can write this as
\begin{align*}
(F_{k+1})_{\xbar}(X) &=
\left ( P_{\xbar}(X) - \varphi_{\xbar}(X) \right) \\
&+ \left ( P_{\xbar}(X) \xi_X(P(X)) - \varphi_{\xbar}(X) \xi_X(X+\varphi(X)) \right) \\
&+ \left ( \overline{ P_{X}(X)} \xi_{\xbar} (P(X)) - \overline{1+\varphi_{X}(X) } \xi_{\xbar}(X+\varphi(X)) \right )\\
&= J_1 + J_2 + J_3.
\end{align*}
By \eqref{blemma8eq6b}, we have
\begin{align*}
\av J_1 \av  &= \av P_{\xbar}(X)  - \varphi_{\xbar}(X)\av \\
&< C_{19}e^{-2\Re X},
\end{align*}
for $X \in L$.
Taking advantage of estimates already calculated, by \eqref{blemma8eq2}, \eqref{blemma8eq6a}, \eqref{blemma8eq6b} and Lemma \ref{blemma1a},
\begin{align*}
\av J_2 \av &= \av \left ( P_{\xbar}(X) \xi_X(P(X)) - \varphi_{\xbar}(X) \xi_X(X+\varphi(X)) \right) \av  \\
&= \av P_{\xbar}(X) [ \xi_X(P(X)) - \xi_X ( X+\varphi (X)  ) ]+ \xi_X(X+\varphi(X) ) [ P_{\xbar}(X) -\varphi_{\xbar}(X)  ] \av \\
&< C_{18}C_{15}e^{-2\Re X}+ C_4C_{19}e^{-2\Re X},
\end{align*}
for $X \in L$. Also, by \eqref{blemma8eq3}, \eqref{blemma8eq6}, the calculation for $I_1$ and Lemma \ref{blemma1a}, we have
\begin{align*}
\av J_3 \av &= \av \left ( \overline{ P_{X}(X)} \xi_{\xbar} (P(X)) - \overline{(1+\varphi_{X}(X)) } \xi_{\xbar}(X+\varphi(X)) \right ) \av  \\
&= \av \overline{P_{X}(X)} [ \xi_{\xbar}(P(X)) - \xi_{\xbar} ( X+\varphi (X)  ) ]+ \xi_{\xbar}(X+\varphi(X) ) [ \overline{P_{X}(X)} -\overline{(1+\varphi_{X}(X)}  ] \av \\
&< C_{16}C_{17}e^{-2\Re X} + \left ( \frac{C_4(C_{10}+C_{13})}{2} \right ) e^{-2\Re X},
\end{align*}
for $X \in L$.
The estimates for $J_1,J_2$ and $J_3$ show that
\[ \av (F_{k+1})_{\xbar}(X) \av < C_9' e^{-2 \Re X}\]
for $X \in L$. Hence if $\sigma$ is chosen large enough so that $C_9'e^{(\alpha-2)\sigma} <C_9$, then
\[ \av (F_{k+1})_{\xbar}(X) \av < C_9 e^{-\alpha \Re X},\]
for $X \in L$.
This completes the proof of the proposition.
\end{proof}

\begin{corollary}
\label{blemma9}
There exists a constant $C_{20}>0$ such that the complex dilatation
$\mu_k$ of $F_k$ satisfies, for all $k \geq 1$,
\[ \av \mu _k(X) \av \leq C_{20}e^{-\alpha\Re X}\]
for all $X \in L$.
\end{corollary}

\begin{proof}
This is an immediate corollary of Proposition \ref{blemma8}.
\end{proof}

\subsection{Proof of Proposition \ref{bprop1}}

Choose $\sigma >0$ large enough so that the results of the previous sections hold in the half-plane $L=\{ \Re X > \sigma \}$.
Recall the definition of the functions $\psi_k$ and assume that they are defined in a neighbourhood of infinity
$U=\{ \av z \av > R\}$ where $R > e^{\sigma}$. Recall that under a logarithmic change of variable, we have
$\widetilde{\psi_k} = F_k$.

Write
\[ \psi _k(z) = \prod _{j=1}^k \frac{\psi_j(z)}{\psi_{j-1}(z)},\]
where $\psi_0(z) \equiv 1$. To show that $\psi_k$ converges uniformly on $U$, it is enough to show that $\log \psi_k(z)$
converges uniformly on $U$, where the principal branch of the logarithm is chosen. Then, writing $z=e^X$,
Lemma \ref{blemma4} implies that
\begin{align*}
\av \log \psi_k(z) \av &= \left \av \sum_{j=1}^k ( \log \psi_j(z) - \log \psi_{j-1}(z) ) \right \av \\
&= \left \av \sum _{j=1}^k F_j(X) - F_{j-1}(X) \right \av \\
&\leq \sum_{j=1}^k \left \av F_j(X) - F_{j-1}(X) \right \av \\
&< C_9 \sum_{j=1}^k \exp \{ -\alpha ^j \Re (X) \} \\
&= C_9 \sum_{j=1}^k \av z \av ^{-\alpha ^j},
\end{align*}
for some constant $C_9>0$ and $\alpha \in (1,2)$.
As $k \to \infty$, this clearly converges on $U = \{\av z \av > R\}$. Hence $\psi_k$ converges uniformly on $U$ to $\psi$,
and we may write $\psi (z) = z+ R(z)$.

For the second part of the proposition, we need to show that $R(z) = o(1)$.
We know that $T_k$ converges uniformly to $T$ for $\Re X >\sigma $ (this is just the content
of the first part of the proof). By this fact and by Lemma \ref{blemma3}, we have
\[ \av T(X) \av <C_8 e^{-\alpha \Re X}, \]
for $\Re X > \sigma$. Now, $\widetilde{\psi}(X) = X+T(X)$ and so, using the fact that $z= e^X$,
we have that
\begin{align*}
\av R(z) \av &= \left \av \exp \left ( \log z + T(\log z) \right ) - z \right \av \\
&= \left \av z \left ( \exp T(\log z) -1 \right ) \right \av \\
&\leq \av z \av \left (  \av T(\log z) \av + o( \av T(\log z ) \av) \right ) \\
&\leq \av z \av \left ( C_8 e^{-\alpha \log \av z \av } + o(\av T(\log z )\av) \right ) \\
&= C_8 \av z \av ^{1-\alpha} + o( \av z \av ^{1-\alpha}).
\end{align*}
Since $\alpha \in (1,2)$, we have that $R(z) = o(1)$ for large $\av z\av$. In fact, although the constants $C_j$ may change,
we actually have that $R(z) = O(\av z \av ^{1-\alpha})$ for any $\alpha \in (1,2)$.

\subsection{Proof of Proposition \ref{bprop2}}

As indicated in the construction of $\psi_k$ in the introductory section, each $\psi_k$ is injective on some neighbourhood $U$ of infinity.
Further, Corollary \ref{blemma9} shows that the complex dilatation $\mu_k$
of $\widetilde{\psi_k}$, which is $\psi_k$ in logarithmic coordinates,
satisfies
\begin{equation}
\label{prop2eq1}
\av \mu _k(X) \av \leq C_{20}e^{-\alpha\Re X},
\end{equation}
for $\alpha \in (1,2)$ and all $X \in L$.
Since $\widetilde{\psi_k}(X) = \log \psi (e^X)$, where $z= e^X$,
and $\log, \exp$ are both holomorphic, it follows that
\[ \av \mu_k(X) \av = \av \mu _{\psi_k}(z) \av.\]
Since $\Re X > \sigma$ corresponds to $\av z \av > e^{\sigma}$,
it follows that $\psi_k$ is quasiconformal in a neighbourhood of infinity.
Moreover, \eqref{prop2eq1} shows that $\mu _{\psi_k}(z) \to 0$
as $\av z \av \to \infty$, which means that $\psi_k$
is asymptotically conformal.

By Proposition \ref{bprop1}, $\psi_k$ converges uniformly on $U$ to
a function $\psi$. Since we may assume each $\psi_k$ is $K$-quasiconformal on $U$ for some $K>1$, by the quasiregular Montel's theorem (see \cite{Miniowitz})
it follows that the limit $\psi$ is also $K$-quasiconformal,
and moreover, that $\psi$ is asymptotically conformal.

\section{Proof of Theorem \ref{thm2}}

Recall that $H=h^2$.
Assume that $K,\theta$ are fixed and the quasiconformal map $\psi$ conjugates $f=H+c$ to $H$ in a neighbourhood $U$ of infinity.
Without loss of generality, we can assume that $U=-U$ where $-U = \{ z \in \C : -z \in U\}$.
To prove the theorem, we need to show that the domain of definition of $\psi$ may be extended. To this end, we prove the following lemma, the proof of which contains standard ideas.

\begin{lemma}
\label{s5l1}
Let $V\subset I(f)$ be a connected neighbourhood of infinity with connected complement, satisfying $V=-V$ and such that $f:f^{-1}(V) \to V$ is a two-to-one covering map.
If $\psi$ is defined on $V$, then $\psi$ can be extended to a quasiconformal map defined on $f^{-1}(V)$ which conjugates $f$ to $h^2$.
\end{lemma}

\begin{remark}
If $V=-V$, then since $h(-z)=-h(z)$ and $g(z)=g(-z)$, it is clear that $f^{-1}(V) = -f^{-1}(V)$.
\end{remark}

\begin{proof}
Let $V$ satisfy the hypotheses of the lemma. Let $w \in V$ and $\gamma$ be a curve connecting $w$ to infinity in $V$.
Since $f$ is a two-to-one covering map from $f^{-1}(V)$ onto $V$, then given $z \in f^{-1}(w)$, $\gamma$ lifts to a curve $\gamma'$
connecting $z$ and infinity in $f^{-1}(V)$. We note that since $V \cup \{\infty \}$ is simply connected and $f:f^{-1}(V) \to V$ is a covering map, $f^{-1}(V) \cup \{ \infty \}$ is also simply connected.

Now, $\eta = \psi(\gamma)$ is a curve in $\psi(V)$ connecting $\psi(w)$ and infinity in $\psi(V) \subset I(H)$. Since
$H:H^{-1}(\psi(V)) \to \psi(V)$ is a two-to-one covering, $\eta$ lifts to two curves in $H^{-1}(\psi(V))$, each terminating at one of the
two points of $H^{-1}(\psi(w))$. Since $\psi$ is defined in a neighbourhood of infinity, there is only one of these two curves, say $\eta'$,
which is the image of $\gamma'$ under $\psi$ near infinity. We then define $\psi(z)$ to be the end-point of $\eta'$. Note that the other lift of $\eta$
corresponds to the other pre-image of $w$ under $f$.

In this way, $\psi$ extends to a map $f^{-1}(V) \to H^{-1}(\psi(V))$, with $\psi(z) \in H^{-1}(\psi(f(z)))$.
Since $f$ is continuous, $\psi$ is continuous on $V$ and $H$ is a local homeomorphism away from $0$, the extension of $\psi$ is continuous.
By construction, $\psi$ still satisfies the
conjugacy $H \circ \psi = \psi \circ f$ on its enlarged domain and hence is still locally quasiconformal. To finish the proof of the lemma,
we have to show that $\psi$ is injective.

Suppose this was not the case, and $\psi(z_1) = \psi(z_2)$ for $z_1,z_2 \in f^{-1}(V)$ (and at least one of $z_1,z_2$ must be in $f^{-1}(V) \setminus V$ since $\psi$ is injective in $V$). Then
\[ \psi(f(z_1)) = H(\psi(z_1)) = H(\psi(z_2)) = \psi(f(z_2)),\]
and since $f(z_1),f(z_2) \in V$ and $\psi$ is injective there, we must have $f(z_1) = f(z_2)$. Thus $z_1=-z_2$ and $\psi(z_1) = \psi(-z_1)$.
Since $V=-V$, we obtain a contradiction: choose curves $\pm \gamma$ from $\pm z_1$ to infinity, and then by continuity we have
$\psi(-z) = -\psi(z)$ on $\gamma$.
\end{proof}

To prove part (i) of Theorem \ref{thm2}, observe that if $c \notin I(f)$, then $f:f^{-n}(U) \to f^{1-n}(U)$ is a two-to-one covering map
for any $n \in \N$. Applying Lemma \ref{s5l1} repeatedly to $f^{-n}(U)$ for $n \in \N$ and noting that
\[ I(f) = \bigcup _{n \geq 1} f^{-n}(U)\]
shows that $\psi $ can be extended to all of $I(f)$. The extension of $\psi$ to $f^{-n}(U)$ is a quasiconformal map, but the distortion
may increase as $n$ increases. Hence we can only conclude that $\psi:I(f) \to I(h^2)$ is an injective locally quasiconformal map.

For part (ii) of Theorem \ref{thm2}, the same reasoning applies as in part (i), but here we can only apply Lemma \ref{s5l1} finitely many times,
since $c \in I(f)$.
That is, once $c \in f^{-n}(U)$, then $f:f^{-(n+1)}(U) \to f^{-n}(U)$ is no longer a two-to-one covering map and we cannot apply Lemma \ref{s5l1}.
However, $\psi$ can be extended to a neighbourhood of infinity which contains $c$, which completes the proof of the theorem.

\section{Proof of Theorem \ref{thm3a}}

\subsection{Fixed rays of $h^2$}

Denote by $R_{\varphi}$ the ray $\{ te^{i\varphi} : t\geq 0\}$.
Let the ray $R_{\phi}$ with angle $\phi = \phi(K,\theta)$ be a fixed ray of $H$, recalling section 2 or Theorem 6.4 of \cite{FG}.

Let $\mu_n$ be the complex dilatation of $H^n$. Then the formula for the complex dilatation of a composition (see for example \cite{FM}) gives
\[ \mu_n (z) = \frac{ \mu_{1}(z) + r_H \mu_{n-1}(H(z)) }{1+r_H \overline{\mu_1(z)}\mu_{n-1}(H(z)) },\]
where $r_H = \overline{H_z(z)}/H_z(z)$.
Now, $\mu_1$ is constant in $\C$, and the next lemma shows that $\mu_n$ is a constant on the fixed ray $R_{\phi}$.

\begin{lemma}
\label{s6l2}
Let $z \in R_{\phi}$. Then for $n \geq 1$
\begin{equation*}
\mu_n(z) \equiv \frac{\mu_1 + e^{-i\phi}\mu_{n-1} }{1+e^{-i\phi}\overline{\mu_1}\mu_{n-1}}.
\end{equation*}
\end{lemma}

\begin{proof}
To find $r_H$, we observe that
\[ H_z(z) = \left [ h(z)^2 \right ]_z = 2(h_z(z))h(z) = (K+1)h(z).\]
Since $z \in R_{\phi}$, we have $z=re^{i\phi}$ for some $r>0$.
By the fact that $R_{\phi}$ is a fixed ray of $H$, it follows that
$h(z) = r'e^{i\phi/2}$ for some $r'>0$. Therefore
\[ r_H(z) = e^{-i\phi}\]
for $z \in R_{\phi}$.
Since $\mu_1 \equiv e^{2i\theta}(K-1)/(K+1)$, by induction we see that
$\mu_n$ is a constant on $R_{\phi}$ and takes the claimed form by the formula for the complex dilatation of a composition.
\end{proof}

We will also need the following lemma.

\begin{lemma}
\label{s6l3}
Any fixed ray $R_{\phi}$ of $H$ lies in the half plane
\[ \mathbb{H}_\theta=\{R_\varphi \;|\; -\pi/2 < \varphi - \theta < \pi/2\}, \]
or if $\theta=\pi/2$ then $R_0$ is the only fixed ray.
\end{lemma}

\begin{proof}
Recall that our normalization for $\theta$ requires $\theta \in (-\pi/2,\pi/2]$.

Let $\pi/2>\theta>0$. First we consider the segment of rays $Q_+$ satisfying,
\[ Q_+ = \{R_\varphi \;|\; \pi > \varphi-\theta > \pi/2\}. \]
Consider where $Q_+$ is mapped to under $H$,
\[ H(Q_+)=\{R_\varphi \;|\; 0>\varphi-2\theta>-\pi\}. \]
We notice that $Q_+\cap H(Q_+)=\emptyset$ and so there can be no fixed
ray in the segment $Q_+$. Next we consider the segment of rays $Q_-$
satisfying,
\[ Q_- = \{R_\varphi \;|\; -\pi > \varphi-\theta > -\pi/2\}. \]
For simplicity we will consider rays to have angle between $-2\pi$ and $0$. Now
\[ H(Q_-)=\{R_\varphi \;|\; -\pi> \varphi-2\theta > -2\pi\}. \]
Recalling that $0<\theta<\pi/2$; we have $H(Q_-)\cap Q_- \neq
\emptyset$, so it is possible that there is a fixed ray in $Q_-$.
However notice that $h(Q_-)=Q_-$ and that for $R_\varphi \in Q_-$ if
$R_\psi=h(R_\varphi)$ then $-\pi<\psi<\varphi<0$. Squaring doubles the
angle so if $R_\tau=H(R_\varphi)$ the angles must satisfy;
\[ -2\pi<\tau<\psi<\varphi<0. \]
This holds for all $R_\varphi\in Q_-$ and so there can be no fixed ray in $Q_-$.

If $-\pi/2<\theta<0$ then analogous arguments work to show $Q_-\cap
H(Q_-)=\emptyset$. Further, if we consider rays to have angle $\varphi \in [0,2\pi]$ then if $R_\varphi\in Q_+$ and $R_\tau=H(R_\varphi)$, we have
$0<\varphi<\tau<2\pi$.

Finally, the rays $R_{\theta\pm\pi/2}$ and $R_{\theta +\pi}$ are fixed
when $\theta=\pm\pi/2$ or $\theta=\pi$ respectively. For our normalization, this only leaves
the case $\theta=\pi/2$; from the discussion
earlier $\phi = 0$ is the only fixed ray for any value of $K$. Also if
$\theta=0$ then $Q_\pm\cap H(Q_\pm)=\emptyset$; completing all
possible cases.
\end{proof}

\subsection{M\"{o}bius transformations}

Define
\[ A(z) = \frac{ \mu_1 + e^{-i\phi}z} {1+e^{-\phi}\overline{\mu_1} z}\]
so that $\mu_n = A^n(\mu_1)$ on the fixed ray $R_{\phi}$.
Note that $A$ depends only on $K,\theta$.
We can rewrite $A$ as
\begin{equation}
\label{moba}
A(z) = e^{-i\phi} \left ( \frac{z+e^{i\phi}\mu _1}{1+ \overline{e^{i\phi}\mu_1}z} \right ).
\end{equation}
Now $A$ is a M\"{o}bius map of the disk $\D$, and the behaviour of the iterates is determined by the trace.
By standard theory, see for example \cite{A}, if $\tr (A)^2 \geq 4$,
then $A$ has all of its fixed points on $\partial \D$ and
$\av A^n(z) \av \to 1$ for all $z \in \D$.
In particular, we would have $\av A^n (\mu_1) \av \to 1$ and
so $\av \mu_n \av \to 1$.
Therefore to prove Theorem \ref{thm3a}, we need to prove the following proposition.

\begin{proposition}
\label{s6p1}
Given the M\"{o}bius transformation $A$ as in \eqref{moba}, we have $\tr (A)^2 \geq 4$.
\end{proposition}

\subsection{Proof of Proposition \ref{s6p1}}

The rest of this section is devoted to proving the proposition. We first calculate an expression for $\tr(A)^2$.

\begin{lemma}
\label{s6l1}
The trace of $A$ satisfies
\[ \tr (A)^2 = \frac{ (K+1)^2(1+\cos \phi)}{2K}.\]
\end{lemma}

\begin{proof}
To compute the trace of a M\"{o}bius transformation $(az+b)/(cz+d)$, we first need to ensure that $ad-bc=1$, and then calculate $a+d$. Putting $A$ into this normalized form yields
\[ A(z) = \frac{ e^{-i\phi/2}\left( \frac{K+1}{2K^{1/2}} \right ) z
+ \mu_1 e^{i\phi/2}\left( \frac{K+1}{2K^{1/2}} \right ) }{ e^{-i\phi/2}\left( \frac{K+1}{2K^{1/2}} \right )\overline{\mu_1} z + e^{i\phi/2}\left( \frac{K+1}{2K^{1/2}} \right )} .\]
From this we can calculate that
\[ (\tr A)^2 = \frac{(K+1)^2(e^{i\phi/2}+e^{-i\phi/2})^2}{4K} = \frac{(K+1)^2(1+\cos \phi )}{2K},\]
which proves the lemma.
\end{proof}

To prove Proposition \ref{s6p1} by using Lemma \ref{s6l1} we need to obtain a lower bound on $\cos\phi$, where $\phi$ is the angle of a fixed ray of $H$ corresponding to $K,\theta$. Recall from lemma~\ref{s6l3} that $\phi\in\mathbb{H}_\theta$, so we need only consider rays $R_\varphi$ where $\varphi-\theta\in(-\pi/2,\pi/2)$.
To find a lower bound,
first consider the function
$$G(\varphi) =\varphi - \theta - \tan^{-1}\left(\frac{\tan(\varphi-\theta)}{K}\right).$$
Recalling the polar form of $h$ given in \eqref{hpolar},
and since $h$ maps rays to rays, the function $G$ describes the change in angle undergone by a ray of angle $\varphi$ under $h$.
Clearly $G(\theta)=0$ since $h$ stretches in the direction $e^{i\theta}$. Further, for the fixed ray of $h^2$ with angle $\phi$,
$G(\phi) = \phi /2$.

We want to know how large $G$ can be, that is, how much of an angle can $h$ move a ray through.
This maximum occurs when the derivative $\frac{\partial G}{\partial\varphi}=0$. Calculating the derivative gives
\begin{equation*}
\frac{\partial G}{\partial\varphi} =1-\frac{K}{(K^2-1)\cos^2(\varphi - \theta) +1}.
\end{equation*}
Hence the maximum value of $G$ occurs when
\[ \cos^2(\varphi - \theta)= \frac{1}{K+1}. \]
Since $\varphi - \theta \in (-\pi/2, \pi/2)$, then the maxima of $G$ are attained at
\[ \varphi_{\pm} =\theta \pm \cos^{-1}[(K+1)^{-1/2}], \]
and the values of $G$ attained there are
\begin{equation*}
G_{\pm} := G\left( \varphi_{\pm} \right)= \pm \left ( \cos^{-1}[(K+1)^{-1/2}] - \tan^{-1} \left( \frac{ \tan (\cos^{-1}[(K+1)^{-1/2}])}{K} \right ) \right ).
\end{equation*}
Using these local maxima, if $0<\varphi - \theta<\pi /2$, then
\[ 0 \leq G(\varphi) \leq G_+\leq \pi/2,\]
and in particular if the fixed ray of angle $\phi$  satisfies $0<\phi - \theta < \pi/2$ we have
\[ 1\geq\cos \phi  \geq \cos 2G_-\geq 0 \]
recalling that $G(\phi) = \phi/2$.
On the other hand, if $0<\varphi - \theta<-\pi/2$, then
\[ 0 \geq G(\varphi) \geq G_+\geq -\pi/2\]
and in particular if $0<\phi - \theta<-\pi/2$
\[ 1\geq\cos \phi \geq \cos 2G_+\geq 0.\]
In either case, we have
\begin{equation}
\label{s6eq10}
\cos \phi \geq \cos 2\left ( \cos^{-1}[(K+1)^{-1/2}] - \tan^{-1} \left( \frac{ \tan (\cos^{-1}[(K+1)^{-1/2}])}{K} \right ) \right )\geq 0.
\end{equation}

We can simplify this expression by using standard trigonometric formula and the expressions

\begin{align}
\cos(\tan^{-1} x)&=(1+x^2)^{-1/2},  \label{cosarctan}\\
\sin(\tan^{-1} x)&=x(1+x^2)^{-1/2},  \label{sinarctan}\\
\tan(\cos^{-1} x)&=(1-x^2)^{1/2}/x,   \label{tanarccos}\\
\sin(\cos^{-1} x)&=(1-x^2)^{1/2}.    \label{sinarccos}
\end{align}

First, using (\ref{tanarccos}) and the addition formula for $\cos$, the right hand side of \eqref{s6eq10} is

\begin{align*}
&\cos\left[2\cos^{-1}[(K+1)^{-1/2}] - 2\tan^{-1}\left(\left(\frac{(1-\frac{1}{K+1})^{1/2}}{(K-1)^{-1/2}}\right)/K\right)\right] \\
&=\cos\left[2\cos^{-1}[(K+1)^{-1/2}]) - 2\tan^{-1}(K^{-1/2})\right] \\
&=  \cos(2\cos^{-1}[(K+1)^{-1/2}])\cos(2\tan^{-1}(K^{-1/2}))
+ \sin(2\cos^{-1}[(K+1)^{-1/2}])\sin(2\tan^{-1}(K^{-1/2})).
\end{align*}

Using the double angle formula
and (\ref{cosarctan}),(\ref{sinarctan}) and (\ref{sinarccos}), one can calculate that
\begin{align*}
\cos(2\cos^{-1}[(K+1)^{-1/2}]) &= \frac{1-K}{1+K},\\
\cos(2\tan^{-1}(K^{-1/2})) &= \frac{K-1}{K+1} ,\\
\sin(2\cos^{-1}[(K+1)^{-1/2}]) & = \frac{2K^{1/2}}{K+1},\\
\sin(2\tan^{-1}(K^{-1/2})) & = \frac{2K^{1/2}}{K+1}.
\end{align*}
Therefore, the right hand side of \eqref{s6eq10} is equal to
\[ -\frac{(K-1)^2}{(K+1)^2} \; + \; \frac{4K}{(K+1)^2} = \frac{-K^2 + 6K -1}{(K+1)^2}.\]

In conclusion, we have
\begin{equation}\label{estimate}
\cos\phi \geq \frac{-K^2 + 6K -1}{(K+1)^2}.
\end{equation}

From Lemma \ref{s6l1} and (\ref{estimate}) we have that:
\begin{align*}
\tr (A)^2 &\geq \frac{(K+1)^2}{2K} + \frac{(K+1)^2(-K^2 + 6K -1)}{2K(K+1)^2} \\
&= \frac{K^2 + 2K +1 -K^2 +6K -1}{2K} \\
&= \frac{8K}{2K} \\
&= 4,
\end{align*}
which completes the proof of Proposition \ref{s6p1}.

\end{document}